\documentclass[review,3p,times,12pt]{elsarticle}

\usepackage{amssymb}
\usepackage{latexsym}
\usepackage{amsmath}
\usepackage{amssymb}
\usepackage{pifont}
\usepackage[colorlinks]{hyperref}
\usepackage{amsthm}
\newtheorem{theorem}{Theorem}[section]

\newdefinition{defn}[thm]{Definition}
\newdefinition{rmk}[thm]{Remark}
\newdefinition{ex}[thm]{Example}
\newproof{pf}{Proof}

\journal{Approximation and Computation in Science and Engineering , edited by N.J. Daras and Th.M. Rassias, Springer}

\begin{document}

\begin{frontmatter}

\title{On Hyers-Ulam-Rassias  stability of a Volterra-Hammerstein functional  integral equation}



\author{Sorina Anamaria Ciplea, Nicolaie Lungu, Daniela Marian, Themistocles M. Rassias}

\address{Sorina Anamaria Ciplea, Technical University of Cluj-Napoca, Department of Management and Technology, 28 Memorandumului Street, 400114, Cluj-Napoca, Romania\\
 {sorina.ciplea@ccm.utcluj.ro}}

\address{Nicolaie Lungu, Technical University of Cluj-Napoca, Department of Mathematics, 28 Memorandumului Street, 400114, Cluj-Napoca, Romania\\
{nlungu@math.utcluj.ro}}

\address{Daniela Marian, Technical University of Cluj-Napoca, Department of Mathematics, 28 Memorandumului Street, 400114, Cluj-Napoca, Romania\\ {daniela.marian@math.utcluj.ro}}

\address{Themistocles M. Rassias, National Technical University of Athens, Department of Mathematics, Zografou Campus, 15780 Athens, Greece\\ {trassias@math.ntua.gr}}

\begin{abstract}
The aim of  this paper is to study Hyers-Ulam-Rassias stability for a Volterra-Hammerstein functional  integral equation in three variables via Picard operators.
\end{abstract}


\end{frontmatter}

\section{1. Introduction}
Ulam stability is an important concept in the theory of functional equations. The origin of Ulam stability theory was an open problem formulated by Ulam, in 1940, concerning the stability of homomorphism \cite{Ulam}.
The first partial answer to Ulam's question came within a year, when Hyers \cite{Hyers} proved a stability result, for additive Cauchy equation in Banach spaces.
The first result on  Hyers-Ulam stability of differential equations was given by Obloza  \cite{Obloza}. Alsina and Ger investigated the stability of  differential equations $y'=y$ \cite{AlsinaGer}.
The result of Alsina and Ger were extended by many authors  \cite{CimpeanPopa}, \cite{Jung1}, \cite{JungRassias1}, \cite{JungRassias2}, \cite{JungRassias3}, \cite{PopaPugna}, \cite{PopaRasa1}, \cite{PopaRasa2}, \cite{RezRassias}, \cite{Rus2}, \cite{Takahasi}  to the stability of the first order linear differential equation and linear differential equations of higher order. For a broader study of Hyers-Ulam stability for functional equations the reader is also referred to the following books and papers: \cite{Abdollahpour}, \cite{Brzdek}, \cite{Jung}, \cite{Jung2}, \cite{Jung3}, \cite{Jung4}, \cite{Kannappan},  \cite{Lee}, \cite{Lee2}, \cite{Milovanovic}, \cite{Mortici}, \cite{Park}, \cite{Ulam}.

The first result proved on the Hyers-Ulam stability of partial differential equations
is due to A. Prastaro and Th.M. Rassias \cite{PraRassias}. Some results regarding Ulam-Hyers stability of partial differential equations were given by S.-M. Jung \cite{Jung2},
S.-M. Jung and K.-S. Lee \cite{JungLee}, N. Lungu and S.A Ciplea \cite{LunguCiplea}, N. Lungu and D. Popa \cite{LunguPopa1}, \cite{LunguPopa2}, \cite{LunguPopa3}, N. Lungu,  C. Craciun \cite{LC}, N. Lungu and D. Marian \cite{LunguMar}, D Marian, S.A. Ciplea and N. Lungu \cite{MarianCipleaLungu}, I.A. Rus and N. Lungu \cite{RL}.
 In \cite{BrzdekPopaRasaXu} Brzdek, Popa, Rasa and Xu presented a unified and systematic approach to the field.
 Some recent results regarding stability analysis and their applications were established by H. Khan, A. Khan, T. Abdeljawad and  A. Alkhazzan \cite{Khan1}, A. Khan,  J.F. Gómez-Aguilar, T.S. Khan and H. Khan \cite{Khan2}, H. Khan,  T. Abdeljawad,  M. Aslam, R.A. Khan and A. Khan  \cite{Khan3}, H. Khan, J.F. Gómez-Aguilar, A. Khan and T.S. Khan \cite{Khan4}. Results regarding fixed point theory and the Ulam stability can be found in \cite{BrzCadCipl}.

In this paper we consider the following Volterra-Hammerstein functional
integral equation in three variables%
\begin{align}\label{eq1}
&u\left( x,y,z\right) =g\left( x,y,z,h\left( u\right) \left( x,y,z\right)\right) \\
&+\int_{0}^{x}\!\!\! \int_{0}^{y}\!\!\! \int_{0}^{z}K\left( x,y,z,r,s,t,f_{1}\left(
u\right) \left( r,s,t\right) \right) drdsdt+ \int_{0}^{\infty
}\!\!\! \int_{0}^{\infty }\!\!\! \int_{0}^{\infty }F\left( x,y,z,r,s,t,f_{2}\left(
u\right) \left( r,s,t\right) \right) drdsdt \nonumber
\end{align}
via Picard operators.

The present paper is motivated by a recent paper \cite{NgocThLo} of L.T.P. Ngoc, T.M.
Thuyet and N.T. Long in which is studied the existence of asymptotically
stable solution for a Volterra-Hammerstein integral equation in three
variables.
Equation \eqref{eq1} is a generalization of equation (1.1) from \cite{NgocThLo}.

\section{2. Existence and uniqueness}

In what follows we consider some conditions relative to equation \eqref{eq1}.

Let $\left( E,\left\vert \cdot \right\vert \right) $ be a Banach space and $%
\Delta =\left\{ \left( x,y,z,r,s,t\right) \in \mathbb{R}_{+}^{6}:s\leq x,r\leq
y,t\leq z\right\} $. Let $\tau >0$ and the set%
\begin{equation*}
X_{\tau }:=\left\{ u\in C\left( \mathbb{R}_{+}^{3},E\right) \mid \exists M\left(
u\right) >0:\left\vert u\left( x,y,z\right)\right\vert  e^{-\tau \left( x+y+z\right)
}\leq M\left( u\right), \forall \left( x,y,z\right) \in \mathbb{R}_{+} \right\} .
\end{equation*}%
On $X_{\tau }$ we consider Bielecki's norm%
\begin{equation*}
\left\Vert u\right\Vert _{\tau }:=\underset{x,y,z\in R_{+}}{\sup }\left(
\left\vert u\left( x,y,z\right)\right\vert  e^{-\tau \left( x+y+z\right) }\right) .
\end{equation*}%
It is clear that $\left( X_{\tau },\left\Vert \cdot \right\Vert _{\tau
}\right) $ is a Banch space. In what follows we assume, relative to \eqref{eq1},
the conditions:
\begin{itemize}
\item[(C1)]
 $g\in C\left( \mathbb{R}_{+}^{3}\times E,E\right) ,K\in C\left(
\Delta \times E,E\right) ,F\in C\left( \Delta\times E,E\right) ,h\in
C\left( X_{\tau },X_{\tau }\right) ,f_{1}\in C\left( X_{\tau },X_{\tau
}\right) ,f_{2}\in C\left( X_{\tau },X_{\tau }\right) ;$

\item[(C2)]  there exists $l_{h}>0$ such that%
\begin{equation*}
\left\vert h\left( u\right) \left( x,y,z\right) -h\left( v\right) \left(
x,y,z\right) \right\vert \leq l_{h}\left\Vert u-v\right\Vert_{\tau } e^{\tau \left(
x+y+z\right) },\forall x,y,z\in \mathbb{R}_{+},\forall u,v\in X_{\tau };
\end{equation*}

\item[(C3)] there exists $l_{g}>0$ such that%
\begin{equation*}
\left\vert g\left( x,y,z,e_{1}\right) -g\left( x,y,z,e_{2}\right)
\right\vert \leq l_{g}\left\vert e_{1}-e_{2}\right\vert ,\forall x,y,z\in
\mathbb{R}_{+},\forall e_{1},e_{2}\in E;
\end{equation*}

\item[(C4)] there exists $l_{K}\in C\left( \Delta , \mathbb{R}_{+}\right) $
such that%
\begin{equation*}
\left\vert K\left( x,y,z,r,s,t,e_{3}\right) -K\left(
x,y,z,r,s,t,e_{4}\right) \right\vert \leq l_{K}\left( x,y,z,r,s,t\right)
\left\vert e_{3}-e_{4}\right\vert ,
\end{equation*}
$\forall \left( x,y,z,r,s,t\right) \in
\Delta ,\forall e_{3},e_{4}\in E;$
\item[(C5)] there exists $l_{F}\in C\left( \Delta, \mathbb{R}_{+}\right) $
such that%
\begin{equation*}
\left\vert F\left( x,y,z,r,s,t,e_{5}\right) -F\left(
x,y,z,r,s,t,e_{6}\right) \right\vert \leq l_{F}\left( x,y,z,r,s,t\right)
\left\vert e_{5}-e_{6}\right\vert ,
\end{equation*}
$\forall \left( x,y,z,r,s,t\right) \in
\Delta ,\forall e_{5},e_{6}\in E;$
\item[(C6)] there exists $l_{f_{1}}>0$ and $l_{f_{2}}>0$ such that%
\begin{equation*}
\left\vert f_{1}\left( u\right) \left( r,s,t\right) -f_{1}\left( v\right)
\left( r,s,t\right) \right\vert \leq l_{f_{1}}\left\vert u\left(
r,s,t\right) -v\left( r,s,t\right) \right\vert ,\forall \left( r,s,t\right)
\in \mathbb{R}_{+}^{3},\forall u,v\in X_{\tau },
\end{equation*}%
\begin{equation*}
\left\vert f_{2}\left( u\right) \left( r,s,t\right) -f_{2}\left( v\right)
\left( r,s,t\right) \right\vert \leq l_{f_{2}}\left\vert u\left(
r,s,t\right) -v\left( r,s,t\right) \right\vert ,\forall \left( r,s,t\right)
\in \mathbb{R}_{+}^{3},\forall u,v\in X_{\tau };
\end{equation*}

\item[(C7)] there exists $l_{1}>0$ and $l_{2}>0$ such that%
\begin{equation*}
\int_{0}^{x}\!\!\!\int_{0}^{y}\!\!\!\int_{0}^{z}l_{f_{1}}l_{K}\left( x,y,z,r,s,t\right)
e^{\tau \left( r+s+t\right) }drdsdt\leq l_{1}e^{\tau \left( x+y+z\right)
},\forall \left( x,y,z,r,s,t\right) \in \Delta ,
\end{equation*}%
\begin{equation*}
\int_{0}^{\infty }\!\!\!\int_{0}^{\infty }\!\!\!\int_{0}^{\infty }l_{f_{2}}l_{F}\left(
x,y,z,r,s,t\right) e^{\tau \left( r+s+t\right) }drdsdt\leq l_{2}e^{\tau
\left( x+y+z\right) },\forall \left( x,y,z,r,s,t\right) \in \Delta ;
\end{equation*}

\item[(C8)] $l_{g}l_{h}+l_{1}+l_{2}<1$;

\item[(C9)] \begin{align*}
& \left\vert g\left( x,y,z,h\left( u\right) \left(
x,y,z\right) \right) \right\vert
+\int_{0}^{x}\!\!\int_{0}^{y}\!\!\int_{0}^{z}\left\vert K\left(
x,y,z,r,s,t,f_{1}\left( 0\right) \left( r,s,t\right) \right) \right\vert
drdsdt\\
&+\int_{0}^{\infty }\!\!\!\int_{0}^{\infty }\!\!\!\int_{0}^{\infty }\left\vert
F\left( x,y,z,r,s,t,f_{2}\left( 0\right) \left( r,s,t\right) \right)
\right\vert drdsdt\leq \alpha \exp \left( \tau \left( x+y+z\right) \right) ,
\end{align*}
$\forall \left( x,y,z,r,s,t\right) \in \Delta $;
\item[(C10)] there exists $m>0$ such that
\begin{equation*}
\int_{0}^{\infty}\!\!\!\int_{0}^{\infty}\!\!\!\int_{0}^{\infty}\left[ l_{f_{1}}l_{K}\left(
x,y,z,r,s,t\right) +l_{f_{2}}l_{F}\left( x,y,z,r,s,t\right) \right]
drdsdt\leq m,\forall \left( x,y,z,r,s,t\right) \in \Delta .
\end{equation*}
\end{itemize}

\begin{theorem}
Under the conditions $\left( C1\right) -\left( C9\right) $ the equation \eqref{eq1}
has in $X_{\tau }$ a unique solution $u^{\ast }.$
\end{theorem}

\begin{proof}
We consider the operator
\begin{equation*}
A:X_{\tau }\rightarrow X_{\tau },A\left( u\right) \left( x,y,z\right) :=%
\text{ second part of \eqref{eq1}}.
\end{equation*}%
First we prove that $A(u)$ maps $X_{\tau }$ in $X_{\tau }.$ For $u\in X_{\tau }$
we have:%
\begin{align*}
&\left\vert A\left( u\right) \left( x,y,z\right) \right\vert  \leq
\left\vert g\left( x,y,z,h\left( u\right) \left( x,y,z\right) \right)
\right\vert +\int_{0}^{x}\!\!\!\int_{0}^{y}\!\!\! \int_{0}^{z}\left\vert K\left(
x,y,z,r,s,t,f_{1}\left( u\right) \left( r,s,t\right) \right) \right\vert
drdsdt \\
&+\int_{0}^{\infty }\!\!\!\int_{0}^{\infty }\!\!\!\int_{0}^{\infty }\left\vert F\left(
x,y,z,r,s,t,f_{2}\left( u\right) \left( r,s,t\right) \right) \right\vert
drdsdt \leq \left\vert g\left( x,y,z,h\left( u\right) \left( x,y,z\right)
\right) \right\vert  \\
&+\int_{0}^{x}\!\!\!\int_{0}^{y}\!\!\!\int_{0}^{z}\left\vert K\left(
x,y,z,r,s,t,f_{1}\left( u\right) \left( r,s,t\right) \right) -K\left(
x,y,z,r,s,t,f_{1}\left( 0\right) \left( r,s,t\right) \right) \right\vert
drdsdt \\
&+\int_{0}^{\infty }\!\!\!\int_{0}^{\infty }\!\!\!\int_{0}^{\infty }\left\vert F\left(
x,y,z,r,s,t,f_{2}\left( u\right) \left( r,s,t\right) \right) -F\left(
x,y,z,r,s,t,f_{2}\left( 0\right) \left( r,s,t\right) \right) \right\vert
drdsdt \\
&+\int_{0}^{x}\!\!\!\int_{0}^{y}\!\!\!\int_{0}^{z}\left\vert K\left(
x,y,z,r,s,t,f_{1}\left( 0\right) \left( r,s,t\right) \right) \right\vert
drdsdt \\
&+\int_{0}^{\infty }\!\!\!\int_{0}^{\infty }\!\!\!\int_{0}^{\infty }\left\vert F\left(
x,y,z,r,s,t,f_{2}\left( 0\right) \left( r,s,t\right) \right) \right\vert
drdsdt .
\end{align*}
We obtain
\begin{align*}
&\left\vert A\left( u\right) \left( x,y,z\right) \right\vert
\leq \left\vert g\left( x,y,z,h\left( u\right) \left( x,y,z\right) \right)
\right\vert +\int_{0}^{x}\!\!\!\int_{0}^{y}\!\!\!\int_{0}^{z}\left\vert K\left(
x,y,z,r,s,t,f_{1}\left( 0\right) \left( r,s,t\right) \right) \right\vert
drdsdt \\
&+\int_{0}^{\infty }\!\!\!\int_{0}^{\infty }\!\!\!\int_{0}^{\infty }\left\vert F\left(
x,y,z,r,s,t,f_{2}\left( 0\right) \left( r,s,t\right) \right) \right\vert
drdsdt \\
&+\int_{0}^{x}\!\!\!\int_{0}^{y}\!\!\!\int_{0}^{z}l_{K}\left( x,y,z,r,s,t\right)
\left\vert f_{1}\left( u\right) \left( r,s,t\right) -f_{1}\left( 0\right) \left( r,s,t\right) \right\vert drdsdt \\
&+\int_{0}^{\infty }\!\!\!\int_{0}^{\infty }\!\!\!\int_{0}^{\infty }l_{F}\left(
x,y,z,r,s,t\right) \left\vert f_{2}\left( u\right) \left( r,s,t\right) -f_{2}\left( 0\right) \left( r,s,t\right)
\right\vert drdsdt \\
&\leq \alpha \exp \left( \tau \left( x+y+z\right) \right)
+\int_{0}^{x}\!\!\!\int_{0}^{y}\!\!\!\int_{0}^{z}l_{K}\left( x,y,z,r,s,t\right)
l_{f_{1}}\left\vert u\left( r,s,t\right) \right\vert e^{\tau \left(
r+s+t\right) }\cdot e^{-\tau \left( r+s+t\right) }drdsdt \\
&+\int_{0}^{\infty }\!\!\!\int_{0}^{\infty }\!\!\!\int_{0}^{\infty }l_{F}\left(
x,y,z,r,s,t\right) l_{f_{2}}\left\vert u\left( r,s,t\right) \right\vert
e^{\tau \left( r+s+t\right) }\cdot e^{-\tau \left( r+s+t\right) }drdsdt \\
&\leq \alpha \exp \left( \tau \left( x+y+z\right) \right) +\left\Vert
u\right\Vert _{\tau }\int_{0}^{x}\!\!\!\int_{0}^{y}\!\!\!\int_{0}^{z}l_{K}\left(
x,y,z,r,s,t\right) l_{f_{1}}e^{\tau \left( r+s+t\right) }drdsdt \\
&+\left\Vert u\right\Vert _{\tau }\int_{0}^{\infty }\!\!\!\int_{0}^{\infty
}\!\!\!\int_{0}^{\infty }l_{F}\left( x,y,z,r,s,t\right)l_{f_{2}} e^{\tau \left(
r+s+t\right) }drdsdt
\end{align*}%
We have%
\begin{equation*}
\left\vert A\left( u\right) \left( x,y,z\right) \right\vert \leq \left[
\alpha +\left\Vert u\right\Vert _{\tau }\left( l_{1}+l_{2}\right) \exp \left( \tau \left( x+y+z\right)
\right) \right] ,
\end{equation*}%
hence $A(u)\in X_{\tau }.$

The operator $A$ is a contraction in $X_{\tau }$ with respect to $\left\Vert
\cdot \right\Vert _{\tau }.$ Indeed, for $u,v\in X_{\tau },$  we have:%
\begin{align*}
&\left\vert A\left( u\right) \left( x,y,z\right) -A\left( v\right) \left(
x,y,z\right) \right\vert  \leq \left\vert g\left( x,y,z,h\left( u\right)
\left( x,y,z\right) \right) -g\left( x,y,z,h\left( v\right) \left(
x,y,z\right) \right) \right\vert   \notag \\
&+\int_{0}^{x\!\!\!}\int_{0}^{y}\!\!\!\int_{0}^{z}\left\vert K\left(
x,y,z,r,s,t,f_{1}\left( u\right) \left( r,s,t\right) \right) -K\left(
x,y,z,r,s,t,f_{1}\left( v\right) \left( r,s,t\right) \right) \right\vert
drdsdt  \notag \\
&+\int_{0}^{\infty }\!\!\!\int_{0}^{\infty }\!\!\!\int_{0}^{\infty }\left\vert F\left(
x,y,z,r,s,t,f_{2}\left( u\right) \left( r,s,t\right) \right) -F\left(
x,y,z,r,s,t,f_{2}\left( v\right) \left( r,s,t\right) \right) \right\vert
drdsdt  \notag \\
&\leq l_{g}\left\vert h\left( u\right) \left( x,y,z\right) -h\left( v\right)
\left( x,y,z\right) \right\vert \notag \\
&+\int_{0}^{x}\int_{0}^{y}\int_{0}^{z}l_{K}\left( x,y,z,r,s,t\right)
l_{f_{1}}\left\vert f_{1}\left( u\right) \left( r,s,t\right) -f_{1}\left(
v\right) \left( r,s,t\right) \right\vert drdsdt \notag\\
&+\int_{0}^{\infty
}\int_{0}^{\infty }\int_{0}^{\infty }l_{F}\left( x,y,z,r,s,t\right)
l_{f_{2}}\left\vert f_{2}\left( u\right) \left( r,s,t\right) -f_{2}\left(
v\right) \left( r,s,t\right) \right\vert drdsdt \notag \\
&\leq l_{g}l_{h}\left\Vert u-v\right\Vert_{\tau } e^{\tau \left( x+y+z\right)
}+\int_{0}^{x}\!\!\!\int_{0}^{y}\!\!\!\int_{0}^{z}l_{K}\left( x,y,z,r,s,t\right)
l_{f_{1}}\left\Vert u-v\right\Vert_{\tau } e^{\tau \left( r+s+t\right) }drdsdt
\notag \\
&+\int_{0}^{\infty }\!\!\!\int_{0}^{\infty }\!\!\!\int_{0}^{\infty }l_{F}\left(
x,y,z,r,s,t\right) l_{f_{2}}\left\Vert u-v\right\Vert_{\tau } e^{\tau \left(
r+s+t\right) }drdsdt \\
&\leq l_{g}l_{h}\left\Vert u-v\right\Vert_{\tau } e^{\tau \left( x+y+z\right)
}+l_{1}\left\Vert u-v\right\Vert_{\tau } e^{\tau \left( x+y+z\right)
}+l_{2}\left\Vert u-v\right\Vert_{\tau } e^{\tau \left( x+y+z\right) }.
\end{align*}%
Then we have:%
\begin{equation*}
\left\Vert A\left( u\right) -A\left( v\right) \right\Vert_{\tau } \leq \left(
l_{g}l_{h}+l_{1}+l_{2}\right) \left\Vert u-v\right\Vert_{\tau }
\end{equation*}%
for all $u,v\in X_{\tau }.$ From $\left( C8\right) $ we have that $A$ is a
contraction. Hence $A$ is a $c-$Picard operator, with%
\begin{equation*}
c=\frac{1}{1-l_{g}l_{h}-l_{1}-l_{2}}.
\end{equation*}%
Hence the equation \eqref{eq1} has a unique solution in $X_{\tau }.$
\end{proof}

\section{Hyers-Ulam-Rassias stability}

In what follows we consider the equation%
\begin{align}\label{eq10}
&u\left( x,y,z\right) =g\left( x,y,z,h\left( u\right) \left( x,y,z\right)\right) \\
&+\int_{0}^{x}\!\!\! \int_{0}^{y}\!\!\! \int_{0}^{z}K\left( x,y,z,r,s,t,f_{1}\left(
u\right) \left( r,s,t\right) \right) drdsdt+ \int_{0}^{\infty
}\!\!\! \int_{0}^{\infty }\!\!\! \int_{0}^{\infty }F\left( x,y,z,r,s,t,f_{2}\left(
u\right) \left( r,s,t\right) \right) drdsdt \nonumber
\end{align}
and the inequality%
\begin{align}\label{eq20}
&\left\vert u\left( x,y,z\right) -g\left( x,y,z,h\left( u\right) \left(
x,y,z\right) \right) -\int_{0}^{x}\!\!\!\int_{0}^{y}\!\!\!\int_{0}^{z}K\left(
x,y,z,r,s,t,f_{1}\left( u\right) \left( r,s,t\right) \right)
drdsdt-\right. \nonumber\\
&-\left.\int_{0}^{\infty }\!\!\!\int_{0}^{\infty }\!\!\!\int_{0}^{\infty }F\left(
x,y,z,r,s,t,f_{2}\left( u\right) \left( r,s,t\right) \right)
drdsdt\right\vert \leq \varphi \left( x,y,z\right) ,
\end{align}
where $\left( E,\left\vert \cdot \right\vert \right) $ is a Banach space and
$\varphi \in C\left( \left[ 0,a\right) ^{3}, \mathbb{R}_{+}\right) $ is increasing$,g\in
C\left( \left[ 0,a\right) ^{3}\times E,E\right) ,K\in C\left( \left[
0,a\right) ^{6}\times E,E\right) ,F\in C\left( \left[ 0,a\right) ^{6}\times
E,E\right) ,h\in C\left( X_{\tau },X_{\tau }\right) ,f_{1}\in C\left(
X_{\tau },X_{\tau }\right) ,f_{2}\in C\left( X_{\tau },X_{\tau }\right) .$

\begin{theorem}
Under the conditions $\left( C1\right) -\left( C10\right) $  and
\begin{itemize}
\item[(i)]
there exists $N>0$ such that%
\begin{equation*}
\left\vert h\left( u\right) \left( x,y,z\right) -h\left( v\right) \left(
x,y,z\right) \right\vert \leq N\left\vert u\left( x,y,z\right) -v\left(
x,y,z\right) \right\vert ,\forall x,y,z\in \left[ 0,a\right) ,\forall u,v\in
X_{\tau };
\end{equation*}

\item[(ii)] $l_{g}N<1$,
\end{itemize}

if $u$ is a solution of \eqref{eq20} and $u^{\ast }$ is the unique
solution of \eqref{eq10}, we have%
\begin{equation*}
\left\vert u\left( x,y,z\right) -u^{\ast }\left( x,y,z\right) \right\vert
\leq C_{KFghf_{1}f_{2}}\varphi \left( x,y,z\right)
\end{equation*}%
where%
\begin{equation*}
C_{KFghf_{1}f_{2}}\varphi \left( x,y,z\right) =\frac{1}{1-l_{g}N}\exp \left(
\frac{m}{1-l_{g}N}\right) ,
\end{equation*}%
i.e the equation \eqref{eq10} is Hyers-Ulam-Rassias stable.
\end{theorem}

\begin{proof}
We have%
\begin{align*}
&\left\vert u\left( x,y,z\right) -u^{\ast }\left( x,y,z\right) \right\vert \\
&\leq \left\vert u\left( x,y,z\right) -g\left( x,y,z,h\left( u\right)
\left( x,y,z\right) \right) -\int_{0}^{x}\!\!\!\int_{0}^{y}\!\!\!\int_{0}^{z}K\left(
x,y,z,r,s,t,f_{1}\left( u\right) \left( r,s,t\right) \right)
drdsdt\right.\\
&\left.-\int_{0}^{\infty }\!\!\!\int_{0}^{\infty }\!\!\!\int_{0}^{\infty }F\left(
x,y,z,r,s,t,f_{2}\left( u\right) \left( r,s,t\right) \right)
drdsdt\right\vert  \\
&+\left\vert g\left( x,y,z,h\left( u\right) \left( x,y,z\right) \right)
-g\left( x,y,z,h\left( u^{\ast }\right) \left( x,y,z\right) \right)
\right\vert  \\
&+\int_{0}^{x}\!\!\!\int_{0}^{y}\!\!\!\int_{0}^{z}\left\vert K\left(
x,y,z,r,s,t,f_{1}\left( u\right) \left( r,s,t\right) \right) -K\left(
x,y,z,r,s,t,f_{1}\left( u^{\ast }\right) \left( r,s,t\right) \right)
\right\vert drdsdt \\
&+\int_{0}^{\infty }\!\!\!\int_{0}^{\infty }\!\!\!\int_{0}^{\infty }\left\vert F\left(
x,y,z,r,s,t,f_{2}\left( u\right) \left( r,s,t\right) \right) -F\left(
x,y,z,r,s,t,f_{2}\left( u^{\ast }\right) \left( r,s,t\right) \right)
\right\vert drdsdt \\
&\leq \varphi \left( x,y,z\right) +l_{g}\left\vert h\left( u\right) \left(
x,y,z\right) -h\left( u^{\ast }\right) \left( x,y,z\right) \right\vert  \\
&+\int_{0}^{x}\!\!\!\int_{0}^{y}\!\!\!\int_{0}^{z}l_{K}\left( x,y,z,r,s,t\right)
l_{f_{1}}\left\vert f_{1}\left( u\right) \left( r,s,t\right) -f_{1}\left(
u^{\ast }\right) \left( r,s,t\right) \right\vert drdsdt \\
&+\int_{0}^{\infty}\!\!\!\int_{0}^{\infty}\!\!\!\int_{0}^{\infty}l_{F}\left(
x,y,z,r,s,t\right) l_{f_{2}}\left\vert f_{2}\left( u\right) \left(
r,s,t\right) -f_{2}\left( u^{\ast }\right) \left( r,s,t\right) \right\vert
drdsdt.
\end{align*}%
From  conditions $\left( i\right) ,\left( ii\right) $ we have:%
\begin{align*}
&\left\vert u\left( x,y,z\right) -u^{\ast }\left( x,y,z\right) \right\vert
\leq \varphi \left( x,y,z\right) +l_{g}N\left\vert u\left( x,y,z\right)
-u^{\ast }\left( x,y,z\right) \right\vert  \\
&+\int_{0}^{x}\!\!\!\int_{0}^{y}\!\!\!\int_{0}^{z}l_{K}\left( x,y,z,r,s,t\right)
l_{f_{1}}\left\vert u\left( r,s,t\right) -u^{\ast }\left( r,s,t\right)
\right\vert drdsdt \\
&+\int_{0}^{\infty}\!\!\!\int_{0}^{\infty}\!\!\!\int_{0}^{\infty}l_{F}\left(
x,y,z,r,s,t\right) l_{f_{2}}\left\vert u\left( r,s,t\right) -u^{\ast }\left(
r,s,t\right) \right\vert drdsdt \\
&\leq \varphi \left( x,y,z\right) +l_{g}N\left\vert u\left( x,y,z\right)
-u^{\ast }\left( x,y,z\right) \right\vert \\
&+\int_{0}^{\infty}\!\!\!\int_{0}^{\infty}\!\!\!\int_{0}^{\infty}\left(l_{K}\left(
x,y,z,r,s,t\right) l_{f_{1}}+l_{F}\left( x,y,z,r,s,t\right) l_{f_{2}}\right)
\left\vert u\left( r,s,t\right) -u^{\ast }\left( r,s,t\right) \right\vert
drdsdt.
\end{align*}%
Then%
\begin{align*}
&\left( 1-l_{g}N\right) \left\vert u\left( x,y,z\right) -u^{\ast }\left(
x,y,z\right) \right\vert \\
&\leq \varphi \left( x,y,z\right) +\int_{0}^{\infty
}\!\!\!\int_{0}^{\infty }\!\!\!\int_{0}^{\infty }\left( l_{K}\left(
x,y,z,r,s,t\right) l_{f_{1}}+l_{F}\left( x,y,z,r,s,t\right) l_{f_{2}}\right)
\left\vert u\left( r,s,t\right) -u^{\ast }\left( r,s,t\right) \right\vert
drdsdt
\end{align*}%
and we have%
\begin{align*}
&\left\vert u\left( x,y,z\right) -u^{\ast }\left( x,y,z\right) \right\vert
\leq \frac{\varphi \left( x,y,z\right) }{1-l_{g}N}\\
&+\frac{1}{1-l_{g}N}%
\int_{0}^{\infty }\!\!\!\int_{0}^{\infty }\!\!\!\int_{0}^{\infty }\left(
l_{K}\left( x,y,z,r,s,t\right) l_{f_{1}}+l_{F}\left(
x,y,z,r,s,t\right) l_{f_{2}}\right) \left\vert u\left( r,s,t\right) -u^{\ast
}\left( r,s,t\right) \right\vert drdsdt.
\end{align*}%
From Wendorf Lemma \cite{FilatovSha} for unbounded domain it follows that%
\begin{align*}
&\left\vert u\left( x,y,z\right) -u^{\ast }\left( x,y,z\right) \right\vert\\
&\leq \frac{\varphi \left( x,y,z\right) }{1-l_{g}N}\exp \left[ \frac{1}{%
1-l_{g}N}\int_{0}^{\infty }\!\!\!\int_{0}^{\infty }\!\!\!\int_{0}^{\infty }\left(
l_{K}\left( x,y,z,r,s,t\right) l_{f_{1}}+l_{F}\left(
x,y,z,r,s,t\right) l_{f_{2}}\right) drdsdt\right]
\end{align*}%
and we have%
\begin{equation*}
\left\vert u\left( x,y,z\right) -u^{\ast }\left( x,y,z\right) \right\vert
\leq \frac{1}{1-l_{g}N}\exp \left[ \frac{m}{1-l_{g}N}\right] \cdot \varphi
\left( x,y,z\right)
\end{equation*}%
and%
\begin{equation*}
\left\vert u\left( x,y,z\right) -u^{\ast }\left( x,y,z\right) \right\vert
\leq C_{KFghf_{1}f_{2}}\cdot \varphi \left( x,y,z\right)
\end{equation*}%
where%
\begin{equation*}
C_{KFghf_{1}f_{2}}\varphi \left( x,y,z\right) =\frac{1}{1-l_{g}N}\exp \left(
\frac{m}{1-l_{g}N}\right) ,
\end{equation*}%
and the equation \eqref{eq10} is Hyers-Ulam-Rassias stable.
\end{proof}

\bibliographystyle{elsarticle-num}

\end{document}